\documentclass[11pt,reqno]{amsart}
\usepackage{amsmath,amssymb}

 \makeatletter
 \evensidemargin  \oddsidemargin
 \makeatother

\begin{document}
\thispagestyle{empty}
 \setcounter{page}{1}

\begin{center}
{\large\bf The stability of a quadratic type functional equation
with the fixed point alternative\vskip.20in
{\bf  M. Eshaghi Gordji } \\[2mm]

{\footnotesize Department of Mathematics,
Semnan University,\\ P. O. Box 35195-363, Semnan, Iran\\
[-1mm] e-mail: {\tt  madjid.eshaghi@gmail.com}}

{\bf  H. Khodaie  } \\[2mm]

{\footnotesize Department of Mathematics,
Semnan University,\\ P. O. Box 35195-363, Semnan, Iran\\
[-1mm] e-mail: {\tt  khodaie.ham@gmail.com}}}
\end{center}
\vskip 5mm

\noindent{\footnotesize{\bf Abstract.}In this paper, we achieve
the general solution and the generalized Hyers-Ulam-Rassias
stability for the quadratic type functional equation
\begin{align*}&f(x+y+2cz)+f(x+y-2cz)+c^2f(2x)+c^2f(2y)\\&=2[f(x+y)+c^2f(x+z)+c^2f(x-z)+c^2f(y+z)+c^2f(y-z)]
 \hspace {2.6 cm}
\end{align*}
 for fixed integers $c$ with $c\neq0,\pm1$, ~by using the fixed
 point alternative.

 \vskip.10in

  \newtheorem{df}{Definition}[section]
  \newtheorem{rk}[df]{Remark}
   \newtheorem{lem}[df]{Lemma}
   \newtheorem{thm}[df]{Theorem}
   \newtheorem{pro}[df]{Proposition}
   \newtheorem{cor}[df]{Corollary}
   \newtheorem{ex}[df]{Example}

 \setcounter{section}{0}
 \numberwithin{equation}{section}

\vskip .2in

\begin{center}
\section{Introduction}
\end{center}  In 1940, S. M. Ulam [26]
proposed the stability problem for functional equations in the
following question regarding to the stability of group homomorphism:

 Let $(G_1,.)$ be a group and let $(G_2,*)$ be a
metric group with the metric $d(.,.).$ Given $\epsilon >0$, dose
there exist a $\delta
>0$, such that if a mapping $h:{G_1}\longrightarrow{G_2}$ satisfies the
inequality $d(h(x.y),h(x)*h(y)) <\delta,$ for all $x,y\in G_1$,
then there exists a homomorphism $H:G_1\longrightarrow G_2$ with
$d(h(x),H(x))<\epsilon,$ for all $x\in G_1?$ In other words, under
what conditions dose a homomorphism exist near an approximately
homomorphism? Generally, the concept of stability for a functional
equation comes up when we the functional equation is replaced by
an inequality which acts as a perturbation of that equation. D. H.
Hyers [9] answered to the question affirmatively in 1941 so if
$f:{E}\longrightarrow{E'}$ such that
$$\|f(x+y)-f(x)-f(y)\|\leq \delta,$$ for all $x,y\in E,$ and for
some $\delta>0$ where E, $E'$ are Banach spaces; then there exists
a unique additive mapping $T:{E}\longrightarrow{E'}$ such that
$$\|f(x)-T(x)\|\leq \delta,$$
for all $x\in E.$ However if $f(tx)$ is a continuous mapping at
$t$ for each fixed $x\in E$ then T is linear. In 1978, Th. M.
Rassias [22] provided a generalization of Hyers's Theorem, which
allows the Cauchy difference to be unbounded.  This new concept is
known as Hyers-Ulam-Rassias stability of functional equations (see
[1,3], [5-10], [20,21]).

The functional equation $$f(x+y)+f(x-y)=2f(x)+2f(y),\eqno \hspace
{0.5
 cm}(1.1)$$
 is related to symmetric bi-additive function and is called a quadratic functional
 equation naturally and every  solution of the quadratic equation (1.2) is
 said to be a quadratic function particulary. It is well known that a function
 $f$ between two real vector spaces is quadratic if and only if there
 exists a unique symmetric bi-additive function $B$ such that
 $f(x)=B(x,x)$ for all $x$
where $$B(x,y)=\frac{1}{4}(f(x+y)-f(x-y)),\eqno \hspace {0.5
 cm}$$ (see [1,14]).
Skof proved Hyers-Ulam-Rassias stability problem for quadratic
functional equation (1.2) for a class of functions
 $f:A\longrightarrow B$, where A is normed space and B is a Banach
space (see [25]). Cholewa [3] noticed that  Skof's Theorem  is
still true if relevant domain $A$ alters to  an abelian group. In
1992, Czerwik proved the  Hyers-Ulam-Rassias stability of the
equation (1.1) (see [4]) and four years later, Grabiec [7]
generalized the result mentioned above. Jun and Kim [11]
introduced the following functional equation
 $$f(2x+y)+f(2x-y)=2f(x+y)+2f(x-y)+12f(x),\eqno \hspace {0.5
 cm}$$
and established the general solution and the generalized
Hyers-Ulam-Rassias stability for the  functional equation (1.3) in
2002. Obviously,  the $f(x)=ax^3$ satisfies the functional
equation (1.3), which is called a cubic functional equation. Every
solution of the cubic functional equation is said to be a cubic
function. Jun and Kim proved also that  a function
 $f$ between two real vector spaces X and Y is a solution of (1.3) if and only if there
 exists a unique function $C:X\times X\times X\longrightarrow Y$ such that
 $f(x)=C(x,x,x)$ for all $x\in X,$ also, $C$ is symmetric for each
 fixed one variable and is additive for fixed two variables.

 More over, Y-S. Jung and I-S. Chang [13] proved Hyers-Ulam-Rassias stability of
the following new cubic type functional equation
\begin{align*}&f(x+y+2z)+f(x+y-2z)+f(2x)+f(2y)\\&=2[f(x+y)+2f(x+z)+2f(x-z)+2f(y+z)+2f(y-z)]
 , \hspace {3.7 cm}
\end{align*}
 by using the fixed point alternative. Several functional
 equations
 have been investigated in [12,18,19,23,24].

 Now, we introduce the following new quadratic type functional equation for fixed
integers $c$ with $c \neq0,\pm1$,
\begin{align*}&f(x+y+2cz)+f(x+y-2cz)+c^2f(2x)+c^2f(2y)\\&=2[f(x+y)+c^2f(x+z)+c^2f(x-z)+c^2f(y+z)+c^2f(y-z)]
 , \hspace {2.7 cm}(1.2)
\end{align*}

In this paper, we establish the general solution and then the
generalized Hyers-Ulam-Rassias stability problem for (1.2) by
using the fixed point alternative [15] as in [17].

\vskip 5mm
\section{Solution of (1.2)}
   Let $X$ and $Y$ be real vector spaces. We here present the general
   solution of (1.2).

\begin{thm}\label{t2}A function $f:{X}\rightarrow{Y}$ satisfies the functional
equation (1.1) if and only if $f:{X}\rightarrow{Y}$ satisfies the
functional equation (1.2). Therefore, every solution of functional
equations (1.2) is also a quadratic function.

\end{thm}

\begin{proof}Let $f$ satisfy the functional
equation (1.1). Putting $x=y=0$ in (1.1), we get $f(0)=0$. Set
$x=0$ in (1.1) to get $f(-y)=f(y)$. Letting $y=x$ and $y=2x$ in
(1.1), respectively, we obtain that $f(2x)=4f(x)$ and
$f(3x)=9f(x)$ for all $x \in X$. By induction, we lead to
$f(kx)=k^2f(x)$ for all positive integer $k$, since $f$ is even
and $f(0)=0$ so $f(kx)=k^2f(x)$ for any integer $k$. Replacing $x$
and $y$ by $2x+y$ and $2x-y$ in (1.1), respectively, gives
$$f(2x+y)+f(2x-y)_=8f(x)+2f(y),
\eqno \hspace {6.5cm} (2.1)$$ using (1.1) and (2.1), we lead to
$$f(2x+y)+f(2x-y)=f(x+y)+f(x-y)+6f(x). \eqno \hspace {4.5cm}
(2.2)$$Now, replacing $x$ and $y$ by $3x+y$ and $3x-y$ in (1.1),
respectively, then using (1.1), we
have$$f(3x+y)+f(3x-y)=f(x+y)+f(x-y)+16f(x). \eqno \hspace {4.5cm}
(2.3)$$ By using the above method, by induction, we infer that
$$f(ax+y)+f(ax-y)=f(x+y)+f(x-y)+2(a^2-1)f(x). \eqno \hspace {4.2cm}
(2.4)$$
 for all $x, y\in X$ and each positive
integer $a>1$ and for a negative integer $a<-1$, replacing $a$ by
$-a$ one can easily prove the validity of (2.4). Therefore (1.1)
implies (2.4) for any integer $a\neq0,\pm1$. First, it is noted
that Eq. (2.4) also implies the following equation
$$f(bx+y)+f(bx-y)=f(x+y)+f(x-y)+2(b^2-1)f(x), \eqno \hspace {4.2cm}
(2.5)$$
 for all integers $b\neq0,\pm1$.
 Substituting $y$ with $by$ into
(2.5), we get by using the identity $f(bx)=b^2f(x)$,
$$f(x+by)+f(x-by)=b^2f(x+y)+b^2f(x-y)+2(1-b^2)f(x). \eqno \hspace {3.7cm}
(2.6)$$ Replacing $y$ by $by$ in (2.4), we observe that
$$f(ax+by)+f(ax-by)=f(x+by)+f(x-by)+2(a^2-1)f(bx). \eqno \hspace {3.5cm}
(2.7)$$ Hence, according to (2.6) and (2.7), we get
$$f(ax+by)+f(ax-by)=b^2f(x+y)+b^2f(x-y)+2(a^2-b^2)f(x), \eqno \hspace {3cm}
(2.8)$$ for any integers $a$,$b$ with $a \neq0$, $b \neq0$ and $a
\neq\pm b$. From now on, assume that $a \neq0$, $b \neq0$ and $a
\neq\pm b$, putting $x=x+y$ and $y=x-y$ in (2.8) and then by using
the identity $f(2x)=4f(x)$, we have
$$f[(a+b)x+(a-b)y]+f[(a-b)x+(a+b)y]=4b^2(f(x)+f(y))+2(a^2-b^2)f(x+y). \eqno \hspace {0cm}
(2.9)$$ Replacing $x$ and $y$ by $x+abz$ and $y+abz$ in (2.9),
respectively, one gets that
\begin{align*}&f[(a+b)x+(a-b)y+2a^2bz]+f[(a-b)x+(a+b)y+2a^2bz]\\&=4b^2(f(x+abz)+f(y+abz))
+2(a^2-b^2)f(x+y+2abz) . \hspace {3.4cm}(2.10)
\end{align*}
Also, Replacing $x$ and $y$ by $x-abz$ and $y-abz$ in (2.9),
respectively, one gets that
\begin{align*}&f[(a+b)x+(a-b)y-2a^2bz]+f[(a-b)x+(a+b)y-2a^2bz]\\&=4b^2(f(x-abz)+f(y-abz))
+2(a^2-b^2)f(x+y-2abz) . \hspace {3.4cm}(2.11)
\end{align*}
Now, by adding (2.10) and (2.11), we arrive at
\begin{align*}f&[(a+b)x+(a-b)y+2a^2bz]+f[(a+b)x+(a-b)y-2a^2bz]\\&
\hspace {.3cm}+f[(a-b)x+(a+b)y+2a^2bz]+f[(a-b)x+(a+b)y-2a^2bz]
\\&=4b^2(f(x+abz)+f(x-abz)+f(y+abz)+f(y-abz))\\&
\hspace {.3cm}+2(a^2-b^2)(f(x+y+2abz)+f(x+y-2abz)). \hspace
{4.4cm}(2.12)
\end{align*}
On the other hand, we substitute $x=x+abz$ and $y=y-abz$ in (2.9),
we obtain
\begin{align*}&f[(a+b)x+(a-b)y+2ab^2z]+f[(a-b)x+(a+b)y-2ab^2z]\\&=4b^2(f(x+abz)+f(y-abz))
+2(a^2-b^2)f(x+y) . \hspace {4.6cm}(2.13)
\end{align*}
And putting $x=x-abz$ and $y=y+abz$ in (2.9), we get
\begin{align*}&f[(a+b)x+(a-b)y-2ab^2z]+f[(a-b)x+(a+b)y+2ab^2z]\\&=4b^2(f(x-abz)+f(y+abz))
+2(a^2-b^2)f(x+y) . \hspace {4.6cm}(2.14)
\end{align*}
Adding (2.13) to (2.14), we lead to
\begin{align*}f&[(a+b)x+(a-b)y+2ab^2z]+f[(a+b)x+(a-b)y-2ab^2z]\\&
\hspace {.3cm}+f[(a-b)x+(a+b)y+2ab^2z]+f[(a-b)x+(a+b)y-2ab^2z]
\\&=4b^2(f(x+abz)+f(x-abz)+f(y+abz)+f(y-abz))\\&
\hspace {.3cm}+4(a^2-b^2)f(x+y). \hspace {8.6cm}(2.15)
\end{align*}
Now, replacing $z$ by $\frac{a}{b}z$ in (2.15), gives
\begin{align*}f&[(a+b)x+(a-b)y+2a^2bz]+f[(a+b)x+(a-b)y-2a^2bz]\\&
\hspace {.3cm}+f[(a-b)x+(a+b)y+2a^2bz]+f[(a-b)x+(a+b)y-2a^2bz]
\\&=4b^2(f(x+a^2z)+f(x-a^2z)+f(y+a^2z)+f(y-a^2z))\\&
\hspace {.3cm}+4(a^2-b^2)f(x+y). \hspace {8.6cm}(2.16)\hspace
{.2cm}
\end{align*}
 If we compare (2.12) with (2.16), we conclude that
\begin{align*}4&b^2(f(x+abz)+f(x-abz)+f(y+abz)+f(y-abz))\\&
\hspace {.3cm}+2(a^2-b^2)(f(x+y+2abz)+f(x+y-2abz))
\\&=4b^2(f(x+a^2z)+f(x-a^2z)+f(y+a^2z)+f(y-a^2z))\\&
\hspace {.3cm}+4(a^2-b^2)f(x+y). \hspace {8.6cm}(2.17)
\end{align*}
By utilizing from equation (2.8), if $a\neq \pm1$, we get
\begin{align*}f&(x+a^2z)+f(x-a^2z)+f(y+a^2z)+f(y-a^2z)\\&
=a^4(f(x+z)+f(x-z)+f(y+z)+f(y-z))+2(1-a^4)(f(x)+f(y)). \hspace
{1.35cm}(2.18)
\end{align*}
Also, since $a \neq\pm b$ so $ab\neq\pm1$, thus by using (2.8), we
have
\begin{align*}f&(x+abz)+f(x-abz)+f(y+abz)+f(y-abz)\\&
=a^2b^2(f(x+z)+f(x-z)+f(y+z)+f(y-z))+2(1-a^2b^2)(f(x)+f(y)).
\hspace {0.7cm}(2.19)
\end{align*}
Hence, according to (2.17) and (2.18) and (2.19), we obtain that
\begin{align*}f&(x+y+2abz)+f(x+y-2abz)
\\&=2f(x+y)+2a^2b^2(f(x+z)+f(x-z)+f(y+z)+f(y-z))\\&
\hspace {.3cm}-4a^2b^2(f(x)+f(y)), \hspace {8.3cm}(2.20)
\end{align*}
for any integers $a$,$b$ with $a \neq0$, $b \neq0$ and $a \neq \pm
1, \pm b$. Setting $c$ instead of $ab$ in (2.20), we get the
relation for any integers $c$ with $c \neq0,\pm1$
\begin{align*}f&(x+y+2cz)+f(x+y-2cz)
\\&=2f(x+y)+2c^2(f(x+z)+f(x-z)+f(y+z)+f(y-z))-4c^2(f(x)+f(y)), \hspace {1cm}
\end{align*}
which, in view of the identity $f(2x)=4f(x)$, gives
\begin{align*}f&(x+y+2cz)+f(x+y-2cz)+c^2f(2x)+c^2f(2y)\\&=2[f(x+y)+c^2f(x+z)+c^2f(x-z)+c^2f(y+z)+c^2f(y-z)]
 , \hspace {3 cm}
\end{align*}
for all $x,y,z\in X$ and any integers $c$ with $c \neq0,\pm1$
which implies that $f$ satisfies (1.2).

 Conversely, suppose that $f$ satisfies (1.2) for fixed
integers $c$ with $c \neq0,\pm1$. Setting $y=z=0$ in (1.2) to get
$f(2x)=4f(x)$ for all $x\in X$, so we can say $f(4x)=16f(x)$.
Replacing $x$ and $y$ by $cz$ and $cz$, respectively, we obtain
$$f(4cz)=2f(2cz)+2c^2[2f((c+1)z)+2f((c-1)z)]-2c^2f(2cz),\eqno \hspace{3.3cm}$$
by using the identities $f(2x)=4f(x)$ and $f(4x)=16f(x)$, we have
$$c^2[f((c+1)z)+f((c-1)z)]=2(c^2+1)f(cz).\eqno \hspace{5.2cm}(2.21)$$
And replacing $x$ and $y$ by $(c+1)z$ and $(c-1)z$, respectively,
we get
\begin{align*}
f(4cz)=2&f(2cz)+2c^2[f((c+2)z)+f((c-2)z)+2f(cz)]\\&-c^2[f(2(c+1)z)+f(2(c-1)z)],
\hspace {6.3cm}
\end{align*}
now, by using the identities and (2.21), we see that
$$c^2[f((c+2)z)+f((c-2)z)]=2(c^2+4)f(cz).\eqno \hspace{5.1cm}(2.22)$$
By using the above method, by induction, we infer that
$$c^2[f((c+k)z)+f((c-k)z)]=2(c^2+k^2)f(cz),\eqno \hspace{5cm}(2.23)$$
 for all $z\in X$ and each positive integer $k\geq1$ for a
 negative integer $k\leq-1$, replacing $k$ by $-k$ one can easily
 prove validity of (2.23). Therefore, (1.2) implies (2.23) for any
 integer $k\neq0$. Substituting $k$ by $3c$ in (2.23) and dividing
 it by $c^2,$ and then by using the identities to get
 $$f(-cz)=f(cz),\eqno \hspace{6cm}$$ for all $z\in X,$ so $f$ is even. Letting $x=y=0$ in
 (1.2) and then using the evenness of $f$, gives $f(2cz)=4c^2f(z)$
 for all $z\in X$. Set $y=0$ in (1.2) and using the identity
 $f(2x)=4f(x)$ and then applying the evenness of $f$, we deduce
 that
$$f(x+2cz)+f(x-2cz)=2c^2f(x+z)+2c^2f(x-z)+4c^2f(z)+2(1-2c^2)f(x)
 . \hspace {.6 cm}(2.24)$$
Setting $2cx$ instead of $x$ in (2.24) and then using the identity
$f(2cx)=4c^2f(x)$, reduces to
$$f(2cx+z)+f(2cx-z)=2f(x+z)+2f(x-z)-4(1-2c^2)f(x)-2f(z)
 . \hspace {1.4 cm}(2.25)$$
Interchange $x$ and $z$ in (2.24), one gets
$$f(z+2cx)+f(z-2cx)=2c^2f(z+x)+2c^2f(x-z)+4c^2f(x)+2(1-2c^2)f(z)
 , \hspace {.5 cm}(2.26)$$
since $f$ is even so from (2.26), we lead to
$$f(2cx+z)+f(2cx-z)=2c^2f(x+z)+2c^2f(x-z)+4c^2f(x)+2(1-2c^2)f(z)
 . \hspace {.4 cm}(2.27)$$
Thus, from (2.25) and (2.27), we obtain that
$$(2c^2-2)f(x+z)+(2c^2-2)f(x-z)=(4c^2-4)f(x)+(4c^2-4)f(z)
 , \hspace {1.8 cm}(2.28)$$
but since $c\neq0,\pm1$ so from (2.28), we have
$$f(x+z)+f(x-z)=2f(x)+2f(z),\eqno \hspace{5.5cm}$$
for all $x,z\in X$. This completes the proof of the Theorem.
\end{proof}

\begin{thm}\label{t2}[15].(the alternative of fixed point.)
Suppose that we are given a complete generalized metric space
$(\Omega,d)$ and a strictly contractive mapping
$T:\Omega\rightarrow\Omega$ with Lipschitz constant $L$. Then, for
each given $x\in\Omega$, either

$d(T^n x, T^{n+1} x)=\infty~$ for all $n\geq0,$\\
 or other exists a natural number $n_{0}$ such that\\
$\star\hspace{.25cm} d(T^n x, T^{n+1} x)<\infty ~$for all $n \geq n_{0};$ \\
 $\star\hspace{.1cm}$ sequence $\{T^n x\}$ is convergent to a fixed point $y^*$ of $~T$;\\
$\star\hspace{.2cm} y^*$is the unique fixed point of $~T$ in the
set $~\Lambda=\{y\in\Omega:d(T^{n_{0}} x, y)<\infty\};$\\
$\star\hspace{.2cm} d(y,y^*)\leq\frac{1}{1-L}d(y, Ty)$ for all
$~y\in\Lambda.$
\end{thm}
Utilizing the above-mentioned fixed point alternative, we now
obtain our main result, i.e., the generalized Hyers-Ulam-Rassias
stability of the functional equation (1.2).

From this point on, let $X$ be a real vector space and let $Y$ be
a Banach space. Before taking up the main subject, we define the
difference operator $\Delta_f:X\times X\times X \rightarrow Y$ by
\begin{align*}\Delta_f(x,y,z)=f&(x+y+2cz)+f(x+y-2cz)+c^2f(2x)+c^2f(2y)
\\&-2[f(x+y)+c^2f(x+z)+c^2f(x-z)+c^2f(y+z)+c^2f(y-z)], \hspace {1.5cm}
\end{align*}for all $x,y,z \in X$ and each fixed integers $c$
such that $c\neq0,\pm1$ where $f$ is a given $f:X\rightarrow Y.$
\begin{thm}\label{t2}Suppose that $j\in\{-1,1\}$ be fixed, and Let $f:X\rightarrow Y$ a function with
$f(0)=0$ for which there exists a function $\varphi:X\times
X\times X \rightarrow [0,\infty)$ such that
$$\lim_{n\rightarrow\infty}\frac{1}{2^{2nj}}\varphi(2^{nj}x,2^{nj}y,2^{nj}z)=0 ,\eqno \hspace{6cm}(2.29)$$
$$\|\Delta_f(x,y,z)\|\leq \varphi(x,y,z),\eqno \hspace{5.8cm}(2.30)$$
for all $x,y,z \in X$ and. If there exists $L=L(j)<1$ such that
the function
$$x\mapsto\psi(x)=\varphi(\frac{x}{2},0,0),\eqno \hspace{8cm}$$
has the property
$$\psi(x)\leq L~.~2^{2j}~.~\psi(\frac{x}{2^j}),\eqno \hspace{7cm}(2.31)\hspace{.05cm}$$
for all $x \in X,$ then there exists a unique quadratic function
$Q:X\rightarrow Y$ such that
$$\|f(x)-Q(x)\| \leq \frac{L^{\frac{j+1}{2}}}{c^2(1-L)}~\psi(x), \eqno \hspace{5.7cm}(2.32)\hspace{.05cm}$$
for all $x \in X.$
\end{thm}

\begin{proof}
Consider the set
$$\Omega=\{g|~ ~g:X \rightarrow Y,~ g(0)=0\}\eqno \hspace{7.8cm}$$
and introduce the generalized metric on $\Omega,$
$$d(g,h)=d_{\psi}(g,h)=\inf\{K \in (0,\infty):~\|g(x)-h(x)\| \leq K \psi(x),~x \in X\}. \eqno \hspace{2cm}$$
It is easy to see that $(\Omega,d)$ is complete. Now we define a
function $T:\Omega\rightarrow\Omega$ by
$$T~g(x)=\frac{1}{2^{2j}}~~g(2^j x)\eqno \hspace{9.2cm}$$
for all $x \in X.$ Note that for all$~g,h\in \Omega,$
\begin{align*}
d (g,h)< K &~\Rightarrow~~\|g(x)-h(x)\| \leq K\psi(x),~~x\in X,
\\&~\Rightarrow~~\|\frac{1}{2^{2j}}g(2^j x)-\frac{1}{2^{2j}}h(2^j x)\|\leq
\frac{1}{2^{2j}}~K~\psi(2^j x),~~x\in X,
\\&~\Rightarrow~~\|\frac{1}{2^{2j}}g(2^j x)-\frac{1}{2^{2j}}h(2^j x)\|\leq
L~K~\psi(x),~~x\in X,
\\&~\Rightarrow~~d(T~g,T~h)\leq L~K.
 \hspace {7cm}
\end{align*}
Hence, we see that
$$d(T~g,T~h)\leq L~d(g,h),\eqno \hspace{8.6cm}$$
for all $g, h\in \Omega$, that is, $T$ is a strictly self-mapping
of $\Omega$ with the Lipschitz constant $L$. Putting $y=z=0$ in
(2.30), we have
$$\|c^2f(2x)-4c^2f(x)\| \leq \varphi(x,0,0), \eqno \hspace{6.1cm}(2.33)$$
now, by using (2.31) with the case $j=1$, we obtain that
$$\|f(x)-\frac{1}{2^2}f(2x)\| \leq\frac{1}{c^2}\frac{1}{2^2}~\varphi(x,0,0)
=\frac{1}{c^2}\frac{1}{2^2}~\psi(2x)\leq \frac{L}{c^2}~\psi(x),
\eqno \hspace{2.6cm}$$ for all $x\in X$, that is,
$d(f,Tf)\leq\frac{L}{c^2}<\infty$.

If we substitute $x=\frac{x}{2}$ in (2.33) and use (2.31) with the
case $j=-1$, then we see that
$$\|f(x)-2^2f(\frac{x}{2})\| \leq \frac{1}{c^2}~\psi(x), \eqno \hspace{7.3cm}$$
for all $x\in X$, that is, $d(f,Tf)\leq\frac{1}{c^2}<\infty$.

Now, from the fixed point alternative in both cases, it follows
that there exists a fixed point $Q$ of $T$ in $\Omega$ such that
$$Q(x)=\lim_{n\rightarrow\infty}\frac{1}{2^{2nj}}f(2^{nj}x) \eqno \hspace{7cm}(2.34)\hspace{.1cm}$$
for all $x\in X$ since $\lim_{n\rightarrow\infty}d(T^nf,Q)=0.$

Also, if we replace $x,y$ and $z$ by $2^{nj}~x, 2^{nj}~y$ and
$2^{nj}~z$ in (2.30), respectively, and divide by $2^{2nj}$, then it
follows from (2.29) and (2.34) that
$$\|\Delta_Q(x,y,z)\|=
\lim_{n\rightarrow\infty}\frac{1}{2^{2nj}}\|\Delta_f(2^{nj}x,2^{nj}y,2^{nj}z)\|\leq
\lim_{n\rightarrow\infty}\frac{1}{2^{2nj}}\varphi(2^{nj}x,2^{nj}y,2^{nj}z)=0,
\eqno \hspace{0cm}$$ for all $x,y,z\in X$, so $\Delta_Q(x,y,z)=0$.
By Theorem 2.1, the function $Q$ is quadratic.

According to the fixed point alterative, since $Q$ is the unique
fixed point of $T$ in the set
$~\Lambda=\{g\in\Omega:d(f,g)<\infty\},Q$ is the unique function
such that
$$\|f(x)-Q(x)\| \leq K~\psi(x), \eqno \hspace{7.5cm}$$
for all $x\in X$ and $K>0$. Again using the fixed point
alterative, gives
$$d(f,Q) \leq \frac{1}{1-L}d(f,Tf)\leq \frac{L^{\frac{j+1}{2}}}{c^2(1-L)}, \eqno \hspace{5.7cm}$$
thus we conclude that
$$\|f(x)-Q(x)\| \leq \frac{L^{\frac{j+1}{2}}}{c^2(1-L)}~\psi(x), \eqno \hspace{6.3cm}$$
this completes the proof.\end{proof}

From Theorem 2.3, we obtain the following Corollary concerning the
Hyers-Ulam-Rassias stability [22] of the functional equation (1.2).

\begin{cor}\label{t2} Suppose that $j\in\{-1,1\}$ be fixed, and  $p\geq0$ be
given with $p\neq2$. Assume that $\varepsilon\geq0$ is fixed. Let
$f:X\rightarrow Y$  a function such that
$$\|\Delta_f(x,y,z)\|\leq\varepsilon({{\|x\|}^p}+{{\|y\|}^p}+{{\|z\|}^p}), \eqno \hspace {4.5cm}(2.35)\hspace{.1cm} $$
for  all $x,y,z \in X.$ Further, assume that $f(0)=0$ in (2.35)
for the case $p>2$. Then there exists a unique quadratic function
$Q:X \rightarrow Y$ such that
$$\|f(x)-Q(x)\|\leq\frac{j\varepsilon}{c^2(4-2^p)}{{\|x\|}^p}, \eqno \hspace {5.2cm}(2.36)\hspace{.15cm} $$
 for all $x\in X,$ where $pj<2j$.
\end{cor}

\begin{proof}Define
$\varphi(x,y,z)=\varepsilon({{\|x\|}^p}+{{\|y\|}^p}+{{\|z\|}^p})$for
all $x,y,z \in X.$ Then the relation (2.29) is true for $pj<2j$.

Since the inequality
$$\frac{1}{2^{2j}}~\psi(2^j x)=\frac{2^{(p-3)j}}{2^p}~\varepsilon~ {\|x\|}^p
\leq2^{(p-3)j}~\psi(x), \eqno \hspace {4.7cm}$$
 for all $x\in X,$ where $pj<2j$, we see that the inequality
 (2.31) holds with $L=2^{(p-3)j}$. Now from (2.32), yields
 (2.36), which complete the proof of the Corollary.
\end{proof}

The following Corollary is the Hyers-Ulam stability [9] of the
functional equation (1.2).
\begin{cor}\label{t2} Assume that $\delta\geq0$ is fixed. Let
$f:X\rightarrow Y$  a function such that
$$\|\Delta_f(x,y,z)\|\leq\delta,\eqno \hspace {5.3cm} $$
for  all $x,y,z \in X.$ Then there exists a unique quadratic
function $Q:X \rightarrow Y$ such that
$$\|f(x)-Q(x)\|\leq\frac{\delta}{9c^2},\eqno \hspace {4.7cm} $$
holds for all $x\in X.$
\end{cor}

\begin{proof}Letting $p=0$ and $\varepsilon=\frac{\delta}{3}~$and
applying Corollary 2.4.
\end{proof}

{\small


}
\end{document}